\documentclass[11pt]{article}%
\usepackage{amsmath}
\usepackage{amsfonts}
\usepackage{amssymb}
\usepackage{amsthm}
\usepackage{graphicx}
\providecommand{\U}[1]{\protect\rule{.1in}{.1in}}
\newtheorem{theorem}{Theorem}[section]

\newtheorem{conjecture}[theorem]{Conjecture}

\newtheorem{lemma}[theorem]{Lemma}
\newtheorem{proposition}[theorem]{Proposition}

\theoremstyle{definition}
\newtheorem{definition}[theorem]{Definition}

\theoremstyle{remark}
\newtheorem{example}[theorem]{Example}
\newtheorem{remark}[theorem]{Remark}

\numberwithin{equation}{section}
\begin{document}

\title{Liouville transformation, analytic approximation of transmutation operators
and solution of spectral problems}
\author{Vladislav V. Kravchenko, Samy Morelos and Sergii M. Torba\\{\small Departamento de Matem\'{a}ticas, CINVESTAV del IPN, Unidad
Quer\'{e}taro, }\\{\small Libramiento Norponiente No. 2000, Fracc. Real de Juriquilla,
Quer\'{e}taro, Qro. C.P. 76230 MEXICO}\\{\small e-mail: vkravchenko@math.cinvestav.edu.mx, smorelos@math.cinvestav.mx,
storba@math.cinvestav.edu.mx \thanks{The authors acknowledge the support from CONACYT, Mexico via the projects 166141 and 222478.}}}
\maketitle

\begin{abstract}
A method for solving spectral problems for the Sturm-Liouville equation
$(pv^{\prime})^{\prime}-qv+\lambda rv$ $=0$ based on the approximation of the
Delsarte transmutation operators combined with the Liouville transformation is
presented. The problem of numerical approximation of solutions and of eigendata is reduced to approximation of a pair of functions depending on the
coefficients $p$, $q$ and $r$ by a finite linear combination of certain
specially constructed functions related to generalized wave
polynomials introduced in \cite{KKTT}, \cite{KT Transmut}. The method allows
one to compute both lower and higher eigendata with an extreme accuracy.
Several necessary results concerning the action of the Liouville
transformation on formal powers arising in the method of spectral parameter
power series are obtained as well as the transmutation operator for the
Sturm-Liouville operator $\frac{1}{r}\left(  \frac{d}{dx}p\frac{d}%
{dx}-q\right)  $.

\end{abstract}

\section{Introduction}
Consider the linear second order differential equation
\begin{equation}
(pv^{\prime})^{\prime}-qv+\lambda rv=0\label{SL_Intro}%
\end{equation}
where $p$, $q$ and $r$ are reasonably good functions (the precise conditions
imposed on them are specified below) and $\lambda$ is a complex number. In a
constantly increasing number of applications it is necessary to find its
solution for a big set of different values of the spectral parameter $\lambda
$. Preferably the method for solving (\ref{SL_Intro}) is sought to be accurate
and fast. Another desirable feature of the method is the possibility to obtain
the approximate solution in an analytical form. This allows one to study
different qualitative properties of the solution and of related quantities.

Meanwhile the accuracy and the fast computation can be attributes of purely
numerical techniques (we refer to \cite{Pryce} for a recommendable
introduction into the subject), the availability of an analytical form for the
approximate solution is a feature of some asymptotic methods, as, e.g., the
WKB method (see, e.g., \cite{Ishimaru}) or of the spectral parameter power
series (SPPS) method (see, e.g., \cite{KrCV08}, \cite{KrPorter2010},
\cite{KKRosu}, \cite{KrTNewSPPS}). In both the asymptotic and the SPPS methods
there appear natural limitations on the largeness or smallness of the
parameter $\lambda$. In the recent paper \cite{KT AnalyticApprox} a method
offering all the above mentioned advantages: accuracy, speed, analytical form
of solution and, additionally, free of the limitations on the size of
$\lambda$, was developed for equations of the form
\begin{equation}
u^{\prime\prime}-qu+\lambda u=0.\label{Schr_Intro}%
\end{equation}
The method is based on the concept of transmutation operators introduced by
Delsarte in \cite{Delsarte} and later studied in dozens of publications (see
the review \cite{Sitnik} and the books\ \cite{BegehrGilbert}, \cite{Carroll},
\cite{Levitan}, \cite{Marchenko} and \cite{Trimeche}). Several recent results
concerning the transmutation operators made it possible in \cite{KT
AnalyticApprox} to convert them from a purely theoretical tool into an
efficient practical method for solving (\ref{Schr_Intro}) and related spectral problems.

Apparently the availability of a good method for (\ref{Schr_Intro}) signifies
its availability also for (\ref{SL_Intro}). It is well known that these
equations are related by the Liouville transformation. However, application of
the Liouville transformation, first, imposes additional restrictions on the
coefficients $p$, $q$ and $r$ (for example, $p$ and $r$ must be real-valued)
and second, implies the transformation of all the functions involved as well
as of the derivatives of some of them. This may lead to undesirable
limitations and difficulties. The aim of the present paper is to develop the
transmutation method directly for equation (\ref{SL_Intro}) with no necessity
to transform it into (\ref{Schr_Intro}). For this we study the action of the
Liouville transformation on the transmutation operators as well as on the
systems of functions involved, called formal powers, which emerge in relation
with the SPPS method and are the main ingredient in the transmutation method
\cite{KT AnalyticApprox}. We prove that the Liouville transformation maps
formal powers of (\ref{SL_Intro}) into formal powers of (\ref{Schr_Intro})
which gives us the possibility to develop the transmutation method directly
for (\ref{SL_Intro}) and to prove corresponding estimates for the accuracy of
approximation. Moreover, we observe that final representations for approximate
solutions of (\ref{SL_Intro}) do not involve the Liouville transformation and
apparently are not restricted by its applicability. One of the developed numerical
examples confirms this observation and the applicability of the proposed
method under weaker conditions than required by the Liouville transformation.
This is formulated as a conjecture.

In the next Section \ref{Sect Liouville} we remind some known facts on the
Liouville transformation. In Section \ref{Sect Formal Powers} we study the
action of the Liouville transformation on the systems of formal powers. In
Section \ref{Sect Transmutation} we combine the transmutation operator for
(\ref{Schr_Intro}) with the Liouville transformation and obtain a
transmutation operator transforming the differential operator from
(\ref{SL_Intro}) into the operator $\frac{d^{2}}{dx^{2}}$ and study its
properties. This leads to an analytical form for the approximate solution and
corresponding estimates (Theorem \ref{Thm Approxim Sols}). Representations for
approximate derivatives of the solutions are obtained as well. In Section
\ref{Sect Numerical} we propose a computational algorithm based on the
obtained representations. describe its numerical realization and discuss some
test problems. We show that big numbers of eigendata can be found with a
remarkable accuracy and find out that the restrictions imposed by the
Liouville transformation are most likely superfluous.

\section{The Liouville transformation}\label{Sect Liouville}

Let $(A,B)$ denote an interval in $\mathbb{R}$. By $AC_{loc}(A,B)$ we denote
all complex valued functions, absolutely continuous with respect to Lebesgue's
measure on all compact subintervals of $(A,B)$.

\begin{lemma}[\cite{Everitt}] \label{definitionH} Let $p$ and $r:(A,B)\rightarrow\mathbb{R}$
be such that $p,p^{\prime},r,r^{\prime}\in AC_{loc}(A,B)$ and $p(y),r(y)>0$
for all $y\in(A,B).$ If $y_{0}\in(A,B)$ and $x_{0}\in(a,b),$ then the mapping
$l:(A,B)\rightarrow(a,b)$ defined by
\[
l(y):=x_{0}+\int_{y_{0}}^{y}\left\{  r(s)/p(s)\right\}  ^{1/2}%
ds,\quad \text{ for all }y\in(A,B)
\]
has an inverse mapping $l^{-1}:(a,b)\rightarrow(A,B)$, where
\[
a=x_{0}-\int_{A}^{y_{0}}\left\{  r(s)/p(s)\right\}  ^{1/2}%
ds\quad\text{ and }\quad b:=x_{0}+\int_{y_{0}}^{B}\left\{  r(s)/p(s)\right\}  ^{1/2}ds.
\]
\end{lemma}

In what follows we assume that both segments $[A,B]$ and $[a,b]$ (and, hence, the integral $\int_A^B\{r(s)/p(s)\}^{1/2}\,ds$) are finite.

\begin{remark}
If in Lemma \ref{definitionH} we choose $x_{0}=0$ and $y_{0}\in\lbrack A,B]$
such that
\begin{equation}
\int_{A}^{y_{0}}\left\{  r(s)/p(s)\right\}  ^{1/2}ds=\int_{y_{0}}^{B}\left\{
r(s)/p(s)\right\}  ^{1/2}ds \label{symetricinterval}%
\end{equation}
then $l:[A,B]\rightarrow\lbrack-b,b]$. Due to Bolzano's theorem such $y_{0}$ exists.
\end{remark}

According to \cite{Everitt} the following theorem establishes the possibility
of the Liouville transformation under the minimal possible requirements.

\begin{theorem}[\cite{Everitt}] \label{transfdeLiouville} Let the functions $p$ and $r$ satisfy
the conditions of Lemma \ref{definitionH}. Then the Sturm-Liouville
differential equation
\begin{equation}
(p(y)v^{\prime})^{\prime}-q(y)v=-\lambda r(y)v\quad \text{ for all }y\in(A,B),
\label{SL}%
\end{equation}
is related with the Schr\"{o}dinger differential equation
\begin{equation}
u^{\prime\prime}-Q(x)u=-\lambda u,\quad \text{ for all }x\in(a,b) \label{Schr}%
\end{equation}
by the Liouville transformation of the variables $y$ and $v$ into
$x$ and $u$,
\[%
\begin{split}
u(x)  &  :=u(l(y)):=\left\{  p(y)r(y)\right\}  ^{1/4}v(y)\quad \text{ for all }y\in
(A,B)\\
&  :=\left\{  p(l^{-1}(x))r(l^{-1}(x))\right\}  ^{1/4}v(l^{-1}%
(x))\quad \text{ for all }x\in(a,b),
\end{split}
\]
and the coefficient $Q$ is given by the relation
\begin{equation}
Q(x)=r(y)^{-1}q(y)-\left\{  r(y)^{-3}p(y)\right\}  ^{1/4}[p(y)(\left\{
p(y)r(y)\right\}  ^{-1/4})^{\prime}]^{\prime}\quad \text{ for all }y\in(A,B).
\label{Q}%
\end{equation}
\end{theorem}

\begin{remark}
The function $q$ in the last theorem can be complex-valued.
\end{remark}

Denote $\rho(y):=(p(y)r(y))^{1/4}$. The Liouville transformation can be
considered as an operator $L:C[A,B]\rightarrow C[a,b]$ acting according to the rule%
\[
u(x)=L[v(y)]=\rho(l^{-1}(x))v(l^{-1}(x)).
\]

Let us introduce the following notations for the differential operators
\[
A=-\frac{d^{2}}{dx^{2}},\quad B=-\frac{d^{2}}{dx^{2}}+Q(x)\quad\text{ and }\quad C=-\frac
{1}{r(y)}\left(  \frac{d}{dy}\left(p(y)\frac{d}{dy}\right)-q(y)\right)  .
\]

The following proposition summarizes the main properties of the
operator $L$.

\begin{proposition}
\begin{enumerate}
\item The uniform norm of the operator $L$ is $\| L\|=\sup
_{y\in[A,B]}\left\vert \rho(y)\right\vert .$
\item The inverse operator is
defined by $v(y)=L^{-1}[u(x)]=\frac{1}{\rho(y)}u(l(y)).$
\item The equality
\begin{equation}
BL=LC \label{BL=LC}%
\end{equation}
is valid on $C[A,B]$.
\end{enumerate}
\end{proposition}

\begin{proof}
The proof of 1. and 2. is obvious. Let us prove 3. Due to the equality
$\rho(y)v(y)=u(l(y))=u(x)$ with $x=l(y)=x_{0}+\int_{y_{0}}^{y}\left(
r(s)/p(s)\right)  ^{1/2}ds\text{ (see Lemma \ref{definitionH})}$ we obtain
$(\rho v)_{y}=u_{x}l_{y}$ and hence
\begin{equation}
u_{x}=\frac{(\rho v)_{y}}{l_{y}}. \label{ux}%
\end{equation}
Then for the second derivatives we have $(\rho v)_{yy}=u_{xx}l_{y}^{2}%
+u_{x}l_{yy}$ and
\[
u_{xx}=\frac{(\rho v)_{yy}-u_{x}l_{yy}}{l_{y}^{2}}.
\]
Straightforward calculation gives us the equality
\begin{equation}%
-u_{xx}(x)+Q(x)u(x)=-\frac{\rho(y)}{r(y)}\left\{  p(y)v_{yy}(y)+p_{y}%
(y)v_{y}(y)-q(y)v(y)\right\}
\end{equation}
and hence $Bu(x)=\rho(y)Cv(y)$. By definition of $L$ we get $Bu(x)=BL[v(y)]$
and $\rho(y)Cv(y)=\rho(l^{-1}(x))\left[  Cv\right]  (l^{-1}(x))=LC[v(y)]$
which proves\ (\ref{BL=LC}).
\end{proof}

\section{Formal powers}\label{Sect Formal Powers}

In \cite{KrCV08} (see also \cite{KrPorter2010} and \cite{APFT}) a
representation for solutions of the Sturm-Liouville equation in the form of a
spectral parameter power series (SPPS) was obtained for which the construction
of certain systems of functions called formal powers is essential. We
introduce such systems for (\ref{Schr}) and for (\ref{SL}) and establish a
relation between them under the Liouville transformation.

\begin{definition}[\cite{KrCV08}] \label{casopq} Let $f\in C^{2}(a,b)\cap C^{1}[a,b]$ be a complex
valued solution of the equation%
\begin{equation}
f^{\prime\prime}-Q(x)f=0 \label{Schrhom}%
\end{equation}
such that $f(x)\neq0$ for any $x\in\lbrack a,b]$. The interval $(a,b)$ is
supposed to be finite. Let us consider the two systems of auxiliary functions $\bigl\{\widetilde{X}^{(n)}\bigr\}_{n=0}^\infty$ and $\bigl\{X^{(n)}\bigr\}_{n=0}^\infty$ defined recursively as follows
\begin{align}
\widetilde{X}^{(0)}(x)& \equiv X^{(0)}(x)\equiv1, \label{X1}\\
\widetilde{X}^{(n)}(x)&=n \int_{x_{0}}^{x} \widetilde{X}^{(n-1)}(s)\left(  f^{2}(s)\right)  ^{(-1)^{n-1}}\,\mathrm{d}s,\label{X2}\\
X^{(n)}(x)&=n\int_{x_{0}}^{x}X^{(n-1)}(s)\left(  f^{2}(s)\right)  ^{(-1)^{n}}\,\mathrm{d}s,\label{X3}
\end{align}
where $x_{0}$ is an arbitrary fixed point in $[a,b].$ Then the system of functions called formal
powers corresponding to (\ref{Schr}) is defined for any $k\in\mathbb{N}\cup\{0\}$ by the relations 
\[
\varphi_{k}(x)=\begin{cases}
f(x)X^{(k)}(x), & k \text{ odd,}\\
f(x)\widetilde{X}^{(k)}(x), & k \text{ even}.
\end{cases}
\]
\end{definition}

The other ``half'' of the recursive integrals
$\widetilde{X}^{(n)}$ and $X^{(n)}$ are used to define another system of
functions
\[
\psi_{k}(x)=\begin{cases}
\frac{1}{f(x)}X^{(k)}(x), & k \text{ even,}\\
\frac{1}{f(x)}\widetilde{X}^{(k)}(x), & k \text{ odd.}
\end{cases}
\]

\begin{remark}
Definition \ref{casopq} requires the existence of a nonvanishing complex
valued solution of (\ref{Schrhom}). In the case when $Q$ is a continuous real
valued function on $[a,b]$, (\ref{Schrhom}) possesses two linearly independent
regular, real-valued solutions\/ $f_{1}$ and $f_{2}$ whose zeros alternate.
Hence one may choose $f=f_{1}+if_{2}$, and this solution has no zeros in
$[a,b]$. If $Q$ is a continuous complex valued function on $[a,b]$ one can
guarantee the existence of a nonvanishing solution \cite[Remark 5]{KrPorter2010}. Let us note that for the construction of the system of formal powers the knowledge of a nonvanishing solution is not
strictly necessary. When $f$ possesses zeros the system of formal powers can
be constructed following the procedure from \cite{KrTNewSPPS}.
\end{remark}

Analogously, let us introduce a system of formal powers corresponding to equation
(\ref{SL}).

\begin{definition}[\cite{KrPorter2010}] \label{casopqr} Let $g,p,r:[A,B]\rightarrow\mathbb{C}$ be
functions such that $g^{2}r$ and $1/\left(  g^{2}p\right)  $ are continuous on
$[A,B]$. Then the following two families of auxiliary functions are well
defined
\begin{align*}
\widetilde{Y}^{(0)}(y)&\equiv Y^{(0)}(y)\equiv1,\label{Y1}\\
\widetilde Y^{(k)}(y)&=\begin{cases}
k\int_{y_{0}}^{y}\widetilde Y^{(k-1)}(s)g^{2}(s)r(s)ds, & k \text{ odd,}\\
k\int_{y_{0}}^{y}\widetilde Y^{(k-1)}(s)\frac{1}{g^{2}(s)p(s)}ds, & k
\text{ even,}%
\end{cases}\\
Y^{(k)}(y)&=\begin{cases}
k\int_{y_{0}}^{y}Y^{(k-1)}(s)\frac{1}{g^{2}(s)p(s)}ds, & k \text{ odd,}\\
k\int_{y_{0}}^{y}\widetilde{Y}^{(k-1)}(s)g^{2}(s)r(s)ds, & k \text{ even,}%
\end{cases}
\end{align*}
where $y_{0}$ is an arbitrary fixed point in $[A,B]$ such that $p$ is
continuous at $y_{0}$ and $p(y_{0})\neq0$.
\end{definition}

Now let us assume additionally that the function $g$ is a solution of the
equation
\begin{equation}
(p(y)g^{\prime})^{\prime}-q(y)g=0. \label{SLhom}%
\end{equation}
Then similarly to Definition \ref{casopq} we define the formal powers
associated to equation (\ref{SL}).

\begin{definition}
\label{psik}  Under the conditions of Definition \ref{casopqr} the formal powers associated to equation (\ref{SL}) are defined for any $k\in\mathbb{N}\cup\{0\}$ as follows
\begin{align*}
\Phi_{k}(y)&=\begin{cases}
g(y)Y^{(k)}(y), & k \text{ odd,}\\
g(y)\widetilde{Y}^{(k)}(y), & k \text{ even.}%
\end{cases} &
\Psi_{k}(y)&=\begin{cases}
\frac{1}{g(y)}Y^{(k)}(y), & k \text{ even,}\\
\frac{1}{g(y)}\widetilde{Y}^{(k)}(y), & k \text{ odd.}%
\end{cases}
\end{align*}
\end{definition}

\begin{theorem}
\label{relphikpsik} Let $p,q,r,Q$ be functions satisfying the conditions of
Theorem \ref{transfdeLiouville}. Assume that (\ref{SLhom}) possesses a
particular solution $g$ on $(A,B)$ such that the conditions of Definition
\ref{casopqr} are fulfilled, and hence $f(x):=f(l(y))=\rho(y)g(y)$ is a
particular solution of (\ref{Schrhom}) on $(a,b)$. Then the following
relations are valid%
\begin{equation}
\rho(y)\Phi_{n}(y)=\varphi_{n}(x)\quad \text{ for all }n\in\mathbb{N}\cup\left\{
0\right\}, \label{psin=phin}%
\end{equation}
that is
\[
\varphi_{n}(x)=L\left[  \Phi_{n}(y)\right]  .
\]

\end{theorem}

\begin{proof}
Let us prove first that ${Y}^{(n)}(y)={X}^{(n)}(x),$ for all $n\in
\mathbb{N}\cup\left\{  0\right\}  $. This will give us (\ref{psin=phin}) for
all odd $n$. The proof can be conducted by induction. For $n=0$ the required
equality follows from the corresponding definitions. Assume that ${Y}%
^{(k)}(y)={X}^{(k)}(x)$ for $n=k$. Consider $n=k+1$.\newline i) If $k$ is even
we obtain the following chain of equalities
\[%
\begin{split}
X^{(k+1)}(x)  &  =(k+1)\int_{x_{0}}^{x}X^{(k)}(s)\frac{1}{f^{2}(s)}%
ds=(k+1)\int_{l(y_{0})}^{l(y)}X^{(k)}(s)\frac{1}{f^{2}(s)}ds\\
&  =(k+1)\int_{y_{0}}^{y}X^{(k)}(l(s))\frac{1}{f^{2}(l(s))}\frac{r(s)^{1/2}%
}{p(s)^{1/2}}ds\\
&  =(k+1)\int_{y_{0}}^{y}Y^{(k)}(s)\frac{1}{g^{2}(s)p(s)}ds  ={Y}^{(k+1)}(y).
\end{split}
\]
ii) In the case when $k$ is odd the proof is similar,
\[%
\begin{split}
X^{(k+1)}(x)  &  =(k+1)\int_{x_{0}}^{x}X^{(k)}(s)f^{2}(s)ds=(k+1)\int
_{l(y_{0})}^{l(y)}X^{(k)}(s)f^{2}(s)ds\\
&  =(k+1)\int_{y_{0}}^{y}X^{(k)}(l(s))f^{2}(l(s))\frac{r(s)^{1/2}}{p(s)^{1/2}%
}ds\\
&  =(k+1)\int_{y_{0}}^{y}Y^{(k)}(s)g^{2}(s)r(s)ds  ={Y}^{(k+1)}(y).
\end{split}
\]
Analogously one can prove that $\widetilde{Y}^{(n)}(y)=\widetilde{X}^{(n)}(x)$
for all $n\in\mathbb{N}\cup\left\{  0\right\}  $ which gives us
(\ref{psin=phin}) for all even $n$.
\end{proof}

\begin{remark}
In a similar way the equalities
\[
\frac{1}{\rho(y)}\Psi_{n}(y)=\psi_{n}(x)\quad \text{ for all }n\in\mathbb{N}\cup
\left\{  0\right\}  ,
\]
are proved. That is,
\begin{equation}
\psi_{n}(x)=L\left[  \Psi_{n}(y)/\rho^{2}(y)\right]  \label{psin=Psin}%
\end{equation}
for all $n\in\mathbb{N}\cup\left\{  0\right\}  $.
\end{remark}

\begin{example}
Consider the following equation
\[
v^{\prime\prime}+v^{\prime}-2v=-\lambda v,\quad y\in(0,2).
\]
It can be transformed into the equation
\[
u^{\prime\prime}-\frac{9}{4}u=-\lambda u,\quad x\in(-1,1)
\]
by means of the following Liouville transformation. Since
$p(y)=r(y)=g(y)=e^{y}$ and $q(y)=2e^{y}$,$\ $we obtain $Q(x)=9/4$. Choosing
$y_{0}=1$ and $x_{0}=0$ we have $x=l(y)=y-1$ and $a=x_{0}-1$, $b=x_{0}+1$.
Therefore%
\[
y=x+1,\quad a=-1,\quad b=1\quad\text{and}\quad f(x)=e^{3(x+1)/2}.
\]

Now we can calculate the first and the second formal powers from Theorem
\ref{relphikpsik},%
\begin{align*}
\widetilde{Y}^{(1)}(y)&=\frac{1}{3}\left(  e^{3y}-e^{3}\right); & {Y}^{(1)}(y)&=\frac{1}{3}\left(  e^{-3}-e^{-3y}\right)  \\
\widetilde{Y}^{(2)}(y)&=\frac{2}{9}\left(  e^{3(1-y)}+3y-4\right);& {Y}^{(2)}(y)&=\frac{2}{9}\left(  e^{3(y-1)}-3y+2\right)  ,\\
\widetilde{X}^{(1)}(x)&=\frac{1}{3}\left(  e^{(3x+1)}-e^{3}\right);&
{X}^{(1)}(x)&=\frac{1}{3}\left(  e^{-3}-e^{-3(x+1)}\right)  \\
\widetilde{X}^{(2)}(x)&=\frac{2}{9}\left(  e^{-3x}+3x-1\right);& {X}^{(2)}(x)&=\frac{2}{9}\left(  e^{3x}-3x-1\right)  .
\end{align*}
Then indeed, 
\begin{equation}%
\begin{split}
\varphi_{1}(x)=f(x)X^{(1)}(x) &  =e^{3(x+1)/2}\cdot \frac{1}{3}\left(
e^{-3}-e^{-3(x+1)}\right)  =e^{3y/2}\cdot \frac{1}{3}\left(  e^{-3}-e^{-3y}\right)    \\
&  =e^{y/2}\Phi_{1}(y)=\left\{  p(y)r(y)\right\}  ^{1/4}\Phi_{1}(y),
\end{split}
\end{equation}
and
\begin{equation}%
\begin{split}
\varphi_{2}(x)=f(x)\widetilde{X}^{(2)}(x) &  =e^{3(x+1)/2}\cdot \frac
{2}{9}\left(  e^{-3x}+3x-1\right)  =e^{3y/2}\cdot \frac{2}{9}\left(  e^{3(1-y)}+3(y-1)-1\right)  \\
&  =e^{y/2}\Phi_{2}(y)=\left\{  p(y)r(y)\right\}  ^{1/4}\Phi_{2}(y).
\end{split}
\end{equation}

\end{example}

\section{Transmutation operators and approximate solutions}\label{Sect Transmutation}

We will use the following statements proved in \cite{KKTT}.

\begin{theorem}[\cite{KKTT}] Let $Q$ be a continuous complex valued function of an independent
real variable $x\in\lbrack-b,b]$ and let $f$ be a particular solution
 of (\ref{Schrhom}) such that $f\in C^{2}[-b,b]$, $f\neq0$ on $[-b,b]$ and
$f(0)=1.$ Denote $h:=f^{\prime}(0)\in\mathbb{C}$. Suppose $\mathbf{T}$ is the
operator defined by
\[
\mathbf{T}u(x)=u(x)+\int_{-x}^{x}\mathbf{K}(x,t;h)u(t)dt
\]
with the kernel
\[
\mathbf{K}(x,t;h)=\frac{h}{2}+K(x,t)+\frac{h}{2}\int_{t}^{x}%
(K(x,s)-K(x,-s))ds,
\]
where $K(x;t)$ is a unique solution of the Goursat problem
\[
\left(  \frac{\partial^{2}}{\partial x^{2}}-Q(x)\right)  K(x,t)=\frac
{\partial}{\partial t^{2}}K(x,t),
\]%
\[
K(x,x)=\frac{1}{2}\int_{0}^{x}Q(s)ds,\quad K(x,-x)=0.
\]
Then  $\mathbf{T}$ transforms $x^{k}$ into $\varphi_{k}(x)$ for any
$k\in\mathbb{N}\cup\left\{  0\right\}  $ and
\begin{equation}
B\mathbf{T}w=\mathbf{T}Aw\label{transmutationT}%
\end{equation}
for any $w\in C^{2}[-b,b]$.
\end{theorem}

\begin{theorem}[\cite{KKTT}] The inverse operator $\mathbf{T}^{-1}$ exists and has the form
\begin{equation}
\mathbf{T}^{-1}u(x)=u(x)-\int_{-x}^{x}\mathbf{K}(t,x;h)u(t)dt.
\end{equation}
\bigskip
\end{theorem}

Combining (\ref{BL=LC}) with (\ref{transmutationT}) we obtain the following statement.

\begin{theorem}
Let $x\in\lbrack-b,b]$, $x_{0}=0$ and $y_{0}\in(A,B)$ be such that
(\ref{symetricinterval}) holds. Then the operator $\mathbf{T}^{-1}L$ is a
transmutation operator for the pair $A$ and $C$ on $C^{2}[A,B]$, i.e.,
\begin{equation}
A\mathbf{T}^{-1}L=\mathbf{T}^{-1}LC.
\end{equation}
\end{theorem}

Application of $\mathbf{T}^{-1}$ to (\ref{BL=LC}) and substitution of
$\mathbf{T}^{-1}B$ by $A\mathbf{T}^{-1}$ (due to (\ref{transmutationT})) gives
us the result.

In the following example we calculate the operator $\mathbf{T}^{-1}L$.

\begin{example}\label{Example16}
Consider the operator
\[
Cv(y)=-v_{yy}(y)-\frac{v_{y}(y)}{y}+\left[  \frac{1}{4y^{2}}+\frac{2}%
{(y-\frac{1}{2})^{2}}\right]  v(y),
\]
where $y\in\lbrack1,2]$. Notice that $r(y)=p(y)=y$, $\rho(y)=y^{1/2}$ and
$q(y)=\frac{1}{4y}+\frac{2y}{(y-\frac{1}{2})^{2}}$. Choosing $x_{0}=0$ and
$y_{0}=3/2$ we propose the change of the variable in the form$\;x=l(y)=y-3/2$.
Then $x\in\lbrack-1/2,1/2]$ and $Q(x)=\frac{2}{(x+1)^{2}}$. A particular
solution $f$ of $f^{\prime\prime}-Q(x)f=0$, such that $f(0)=1$ and $h=f^{\prime
}(0)=2$, can be chosen in the form $f(x)=(x+1)^{2}$. In \cite{Kravchenko y Torba} it was
shown that for this particular function the corresponding transmutation
kernels have the form
\[
K(x,t)=\frac{2x+2t+x^{2}-t^{2}}{4(x+1)},\quad\mathbf{K}(x,t;2)=\frac
{6x+4+3x^{2}+2t-3t^{2}}{4(x+1)}.
\]
Then
\begin{equation}%
\begin{split}
\mathbf{T}^{-1}L\left[  v(y)\right]   &  =\mathbf{T}^{-1}[u](x)=u(x)-\int
_{-x}^{x}\mathbf{K}(t,x;h)u(t)dt\\
&  =\rho(y)v(y)-\int_{l(-y+2y_{0})}^{l(y)}\mathbf{K}(t,l(y);h)u(t)dt\\
&  =y^{1/2}v(y)-\int_{-y+2y_{0}}^{y}\mathbf{K}(l(t),l(y);h)u(l(t))dt\\
&  =y^{1/2}v(y)-\int_{-y+2y_{0}}^{y}\mathbf{K}(l(t),l(y);h)t^{1/2}v(t)dt,\\
&  =y^{1/2}v(y)-\int_{-y+3}^{y}\frac{3t^{2}-3t+11y-3y^{2}-8}{4t-2}%
t^{1/2}v(t)dt.
\end{split}
\end{equation}
This is a closed form of the operator transmuting solutions of $Cv=\lambda v$
into solutions of $Aw=\lambda w$. For example, application of the obtained
operator to the function $g(y)=L^{-1}f(x)=\frac{(y-1/2)^{2}}{\sqrt{y}}$ (which
is a null solution of $C$) gives us $\mathbf{T}^{-1}L\left[  g(y)\right]  =1$
(a null solution of $A$). The inverse operator $L^{-1}\mathbf{T}$ can also be
constructed explicitly. We have
\begin{align*}
L^{-1}\mathbf{T}\left[  w(x)\right]   &  =L^{-1}\left[  w(x)+\int_{-x}%
^{x}\mathbf{K}(x,t;h)w(t)dt\right] \\
&  =\frac{1}{\rho(y)}\left(  w(l(y))+\int_{-l(y)}^{l(y)}\mathbf{K}%
(l(y),t;h)w(t)dt\right) \\
&  =\frac{1}{\rho(y)}\left(  w(l(y))+\int_{-l(y)}^{l(y)}\mathbf{K}%
(l(y),t;h)w(t)dt\right)  .
\end{align*}
Its application to $w\equiv1$, indeed, gives us the function $g(y)$.
\end{example}

Due to (\ref{transmutationT}) the operator $\mathbf{T}$ maps solutions of the
equation $Aw=\lambda w$ (linear combinations of $\cos\sqrt{\lambda}x$ and
$\sin\sqrt{\lambda}x$) into solutions of $Bu=\lambda u$.

Together with the transmutation $\mathbf{T}$ it is often convenient to
consider the other two operators enjoying the transmutation property
(\ref{transmutationT}) on subclasses of $C^{2}[-b,b]$ (as well as on subclasses of $C^{2}[0,b]$), for details see 
\cite{Marchenko} and additionally \cite{Kravchenko y Torba},
\[
T_{c}w(x)=w(x)+\int_{0}^{x}\mathbf{C}(x,t)w(t)dt
\]
and%
\[
T_{s}w(x)=w(x)+\int_{0}^{x}\mathbf{S}(x,t)w(t)dt
\]
with the kernels $\mathbf{C}$ and $\mathbf{S}$ related to the kernel
$\mathbf{K}$ by the equalities
\begin{equation}
\mathbf{C}(x,t)=\mathbf{K}(x,t;h)+\mathbf{K}(x,-t;h)\label{Cf}%
\end{equation}
and
\begin{equation}
\mathbf{S}(x,t)=\mathbf{K}(x,t;h)-\mathbf{K}(x,-t;h).\label{Sf}%
\end{equation}
The following statement is valid.

\begin{theorem}
[\cite{Marchenko}]\label{TcTsMapsSolutions} Solutions $c(\omega,x;h)$ and
$s(\omega,x;\infty)$ of the equation
\begin{equation}
-u^{\prime\prime}+Q(x)u=\omega^{2}u,\qquad Q\in C[0,b] \quad (\text{or }Q\in C[-b,0])\label{SchrQ}%
\end{equation}
satisfying the initial conditions
\begin{gather}
c(\omega,0;h)=1,\qquad c_{x}^{\prime}(\omega,0;h)=h\label{ICcos}\\
s(\omega,0;\infty)=0,\qquad s_{x}^{\prime}(\omega,0;\infty)=1 \label{ICsin}%
\end{gather}
can be represented in the form
\begin{equation}
c(\omega,x;h)=\cos\omega x+\int_{0}^{x}\mathbf{C}(x,t)\cos\omega t\,dt
\label{c cos}%
\end{equation}
and
\begin{equation}
s(\omega,x;\infty)=\frac{\sin\omega x}{\omega}+\int_{0}^{x}\mathbf{S}%
(x,t)\frac{\sin\omega t}{\omega}\,dt. \label{s sin}%
\end{equation}
\end{theorem}

Let $f\in C^{2}(-b,b)\cap C^{1}[-b,b]$ be a solution of (\ref{Schrhom})
such that $f(x)\neq0$ for any $x\in\lbrack-b,b]$ and $f(0)=1$, $f^{\prime
}(0)=h\in\mathbb{C}$. Denote
\begin{align}
\mathbf{c}_{0}(x)&=f(x), \label{c0}\\
\mathbf{c}_{m}(x)&=\sum_{\text{even }k=0}^{m}\binom{m}{k}x^{k}\varphi
_{m-k}(x),\quad m=1,2,\ldots \label{cm}\\
\mathbf{s}_{m}(x)&=\sum_{\text{odd }k=1}^{m}\binom{m}{k}x^{k}\varphi
_{m-k}(x),\quad m=1,2,\ldots\label{sm}
\end{align}
where the functions $\varphi_{n}$ are those from Definition \ref{casopq} with
$x_{0}=0$.

In \cite{KT AnalyticApprox} the following result was proved.

\begin{theorem}[\cite{KT AnalyticApprox}]
\label{ThApproxCS}The solutions $c(\omega,x;h)$ and $s(\omega,x;\infty)$ of
equation (\ref{SchrQ}) satisfying (\ref{ICcos}) and (\ref{ICsin}) respectively
can be approximated by the functions
\begin{equation}
c_{N}(\omega,x)=\cos\omega x+2\sum_{n=0}^{N}a_{n}\sum_{\text{even }k=0}%
^{n}\binom{n}{k}\varphi_{n-k}(x)\int_{0}^{x}t^{k}\cos\omega t\,dt \label{cN}%
\end{equation}
and
\begin{equation}
s_{N}(\omega,x)=\frac{1}{\omega}\left(  \sin\omega x+2\sum_{n=1}^{N}b_{n}%
\sum_{\text{odd }k=1}^{n}\binom{n}{k}\varphi_{n-k}(x)\int_{0}^{x}t^{k}%
\sin\omega t\,dt\right)  \label{sN}%
\end{equation}
where the coefficients $\left\{  a_{n}\right\}  _{n=0}^{N}$ and $\left\{
b_{n}\right\}  _{n=1}^{N}$ are such that
\begin{equation}
\left\vert \frac{h}{2}+\frac{1}{4}\int_{0}^{x}Q(s)ds-\sum_{n=0}^{N}%
a_{n}\mathbf{c}_{n}(x)\right\vert \leq\varepsilon_{1} \label{KxxErr}%
\end{equation}
and
\begin{equation}
\left\vert \frac{1}{4}\int_{0}^{x}Q(s)ds-\sum_{n=1}^{N}b_{n}\mathbf{s}%
_{n}(x)\right\vert \leq\varepsilon_{2} \label{KxmxErr}%
\end{equation}
for every $x\in\lbrack-b,b]$, and the following estimates hold%
\begin{equation}
\left\vert c(\omega,x;h)-c_{N}(\omega,x)\right\vert \leq\frac{\varepsilon
\sinh(Cx)}{C} \label{estc2}%
\end{equation}
and%
\begin{equation}
\left\vert s(\omega,x;\infty)-s_{N}(\omega,x)\right\vert \leq\frac
{\varepsilon\sinh(Cx)}{\left\vert \omega\right\vert C} \label{estc4}%
\end{equation}
for any $\omega\in\mathbb{C}$, $\omega\neq0$ belonging to the strip
$\left\vert \operatorname{Im}\omega\right\vert \leq C$, $C\geq0$, where
$\varepsilon\geq0$ depends on $\varepsilon_{1}$, $\varepsilon_{2}$ and $Q$.
\end{theorem}

\begin{remark}\label{RmkNoLiouville}
The approximation problems represented by (\ref{KxxErr}) and (\ref{KxmxErr})
can be written in terms of the variable $y$ and with no reference to equation
(\ref{SchrQ}). Indeed, the following equalities hold
\begin{align*}
\frac{h}{2}+\frac{1}{4}\int_{0}^{x}Q(s)ds &=\frac{h}{2}+\frac{1}{4}\int_{y_{0}%
}^{y}\frac{1}{(pr)^{1/4}}\left\{  \frac{q}{(pr)^{1/4}}-[p\{(pr)^{-1/4}%
\}^{\prime}]^{\prime}\right\}  (s)ds=:G_{1}(y),\\
\frac{1}{4}\int_{0}^{x}Q(s)ds&=\frac{1}{4}\int_{y_{0}}^{y}\frac{1}{(pr)^{1/4}%
}\left\{  \frac{q}{(pr)^{1/4}}-[p\{(pr)^{-1/4}\}^{\prime}]^{\prime}\right\}
(s)ds=:G_{2}(y),\\
\mathbf{c}_{0}(x)&=\widetilde{\mathbf{c}}_{0}(y):=\rho(y)g(y),\\
\mathbf{c}_{m}(x)&=\widetilde{\mathbf{c}}_{m}(y):=\rho(y)\sum_{\text{even }%
k=0}^{m}\binom{m}{k}\left(  l(y)\right)  ^{k}\Phi_{m-k}(y),\quad
m=1,2,\ldots\\
\mathbf{s}_{m}(x)&=\widetilde{\mathbf{s}}_{m}(y):=\rho(y)\sum_{\text{odd }%
k=1}^{m}\binom{m}{k}\left(  l(y)\right)  ^{k}\Phi_{m-k}(y),\quad m=1,2,\ldots
\end{align*}
where the system of functions $\Phi_{n}$ is constructed from a particular
solution $g$ of (\ref{SLhom}) satisfying the initial condition
\begin{equation}\label{PartSolIC}
g(y_{0})=\left(  p(y_{0})r(y_{0})\right)  ^{-1/4}%
\end{equation}
(in this case $f(0)=1$, where $f(x)=L\left[  g(y)\right]  $), and $h$ equals
the value of the following expression in $y_{0}$,
\begin{equation}
h=\sqrt{\frac{p(y_{0})}{r(y_{0})}}\left(  \frac{g^{\prime}(y_{0})}{g(y_{0}%
)}+\frac{\rho^{\prime}(y_{0})}{\rho(y_{0})}\right)  \label{h}%
\end{equation}
(in this case $f^{\prime}(0)=h$).

Thus, the coefficients $\left\{  a_{n}\right\}  _{n=0}^{N}$ and $\left\{
b_{n}\right\}  _{n=1}^{N}$ are such that
\begin{equation}
\left\vert G_{1}(y)-\sum_{n=0}^{N}a_{n}\widetilde{\mathbf{c}}_{n}%
(y)\right\vert \leq\varepsilon_{1} \label{Opt1}%
\end{equation}
and
\begin{equation}
\left\vert G_{2}(y)-\sum_{n=1}^{N}b_{n}\widetilde{\mathbf{s}}_{n}%
(y)\right\vert \leq\varepsilon_{2}, \label{Opt2}%
\end{equation}
for all $y\in[A,B]$.
\end{remark}

\begin{remark}\label{Rmk NoSecondDerivatives}
The expressions for the functions $G_{1}$ and $G_{2}$ involve second
derivatives of the coefficients. It is easy to transform them into a form
requiring first derivatives only. Indeed, besides (\ref{Q}) the potential $Q$
admits the following representation (see, e.g., \cite[p. 141]{Zwillinger})
\[
Q(x)=\frac{q(y)}{r(y)}+\frac{\rho_{xx}}{\rho}.
\]
The integral $\int_{0}^{x}\frac{\rho_{ss}}{\rho}ds$ can be written in the form
(due to the identity $\frac{\rho^{\prime\prime}}{\rho}=\left(  \frac
{\rho^{\prime}}{\rho}\right)  ^{\prime}+\left(  \frac{\rho^{\prime}}{\rho
}\right)  ^{2}$)
\begin{align*}
\int_{0}^{x}\frac{\rho_{ss}}{\rho}ds  &  =\frac{\rho_{x}(l^{-1}(x))}%
{\rho(l^{-1}(x))}-\frac{\rho_{x}(l^{-1}(0))}{\rho(l^{-1}(0))}+\int_{0}%
^{x}\left(  \frac{\rho_{s}}{\rho}\right)  ^{2}ds\\
&  =\frac{1}{4}\left(  P(y)-P(y_{0})\right)  +\frac{1}{16}\int_{0}^{x}%
P^{2}(\tau(s))ds\\
&  =\frac{1}{4}\left(  P(y)-P(y_{0})\right)  +\frac{1}{16}\int_{y_{0}}%
^{y}\frac{p^{1/2}(\tau)}{r^{1/2}(\tau)}P^{2}(\tau)d\tau
\end{align*}
where $P(y):=p^{-1/2}(y)r^{-3/2}(y)\left(  p^{\prime}(y)r(y)+p(y)r^{\prime
}(y)\right)  $. \bigskip
\end{remark}

Let us consider the preimages of the solutions $c(\omega,x;h)$ and
$s(\omega,x;\infty)$ under the Liouville transformation
\begin{equation}
L:C\left[  A,B\right]  \rightarrow C[-b,b]\qquad\text{with}\quad l(y)=\int_{y_0}%
^{y}\left\{  r(s)/p(s)\right\}  ^{1/2}ds, \label{Ltransform}%
\end{equation}%
\[
v_{1}(\omega,y):=L^{-1}[c(\omega,x;h)]\quad\text{and\quad}v_{2}(\omega
,y):=L^{-1}[s(\omega,x;\infty)].
\]
Being solutions of the equation%
\begin{equation}
(p(y)v^{\prime})^{\prime}-q(y)v=-\omega^{2}r(y)v\text{ \quad on }(A,B),
\label{SL AB}%
\end{equation}
they satisfy the initial conditions%
\begin{align}
v_{1}(\omega,y_0)&=\frac{1}{\rho(y_0)},& v_{1}^{\prime}(\omega,y_0)&=-\frac
{\rho^{\prime}(y_0)}{\rho^{2}(y_0)}+\frac{h}{\rho(y_0)}\sqrt{\frac{r(y_0)}{p(y_0)}},
\label{InitCondV1}\\
v_{2}(\omega,y_0)&=0,& v_{2}^{\prime}(\omega,y_0)&=\frac{1}{\rho(y_0)}\sqrt
{\frac{r(y_0)}{p(y_0)}}. \label{InitCondV2}%
\end{align}
Theorem \ref{ThApproxCS} together with Theorem \ref{relphikpsik} allow us to
obtain convenient representations for approximations of $v_{1}$ and $v_{2}$.

\begin{theorem}\label{Thm Approxim Sols}
Let $g$ be a solution of (\ref{SLhom}) satisfying the initial condition
$g(y_0)=\left(  p(y_0)r(y_0)\right)  ^{-1/4}$ such that the conditions of Definition
\ref{casopqr} are fulfilled. Let $v_{1}$ and $v_{2}$ be solutions of
(\ref{SL AB}) satisfying (\ref{InitCondV1}) and (\ref{InitCondV2})
respectively, where $h$ is the complex number defined by (\ref{h}). Let $L$ be the Liouville transformation (\ref{Ltransform}). Then
$v_{1}$ and $v_{2}$ can be approximated by the functions $v_{1,N}$ and
$v_{2,N}$ respectively, defined by the equalities%
\begin{equation}
v_{1,N}(\omega,y)=\frac{1}{\rho(y)}\cos\left(  \omega l(y)\right)
+2\sum_{n=0}^{N}a_{n}\sum_{\text{even }k=0}^{n}\binom{n}{k}\Phi_{n-k}%
(y)\int_{0}^{l(y)}t^{k}\cos\omega t\,dt \label{v1N}%
\end{equation}
and
\begin{equation}
v_{2,N}(\omega,y)=\frac{1}{\omega}\left(  \frac{\sin\left(  \omega
l(y)\right)  }{\rho(y)}+2\sum_{n=1}^{N}b_{n}\sum_{\text{odd }k=1}^{n}\binom
{n}{k}\Phi_{n-k}(y)\int_{0}^{l(y)}t^{k}\sin\omega t\,dt\right),  \label{v2N}%
\end{equation}
where the coefficients $\left\{  a_{n}\right\}  _{n=0}^{N}$ and $\left\{
b_{n}\right\}  _{n=1}^{N}$ are such that (\ref{Opt1}) and (\ref{Opt2}) are
fulfilled. The following estimates hold%
\begin{equation}
\left\Vert v_{1}-v_{1,N}\right\Vert \leq\frac{\varepsilon\rho_{0}\sinh(Cb)}{C}
\label{estv2}%
\end{equation}
and%
\begin{equation}
\left\Vert v_{2}-v_{2,N}\right\Vert \leq\frac{\varepsilon\rho_{0}\sinh
(Cb)}{\left\vert \omega\right\vert C} \label{estv4}%
\end{equation}
for any $\omega\in\mathbb{C}$, $\omega\neq0$ belonging to the strip
$\left\vert \operatorname{Im}\omega\right\vert \leq C$, $C\geq0$, where
$\varepsilon\geq0$ depends on $\varepsilon_{1}$, $\varepsilon_{2}$ and $Q$,
$\rho_{0}:=\left\Vert 1/\rho\right\Vert $ and $\left\Vert \cdot\right\Vert $
denotes the maximum norm on $[A,B]$.
\end{theorem}

\begin{proof}
Observe that $v_{1,N}(\omega,y)=L^{-1}[c_{N}(\omega,x)]$ and $v_{2,N}%
(\omega,y)=L^{-1}[s_{N}(\omega,x)]$. Indeed, application of Theorem
\ref{relphikpsik} gives us the result. The estimates (\ref{estv2}) and
(\ref{estv4}) follow from (\ref{estc2}) and (\ref{estc4}) taking into account
that $\left\Vert L^{-1}\right\Vert =\left\Vert 1/\rho\right\Vert $.
\end{proof}

Solution of problems involving derivatives in boundary conditions requires
convenient approximations for $v_{1}^{\prime}$ and $v_{2}^{\prime}$. Direct
differentiation of (\ref{v1N}) and (\ref{v2N}) does not present any
difficulty, nevertheless it is still necessary to be able to obtain
corresponding estimates for the difference $v^{\prime}-v_{N}^{\prime}$ where
$v$ represents $v_{1}$ or $v_{2}$. In \cite{KT AnalyticApprox} it was shown
that instead of this direct approach one may choose another possibility based
on certain results concerning transmutations for Darboux associated equations
of the form (\ref{Schr}). As a result the approximation of the derivatives of
the solutions $c(\omega,x;h)$ and $s(\omega,x;\infty)$ of (\ref{SchrQ}) is
obtained in the form
\begin{equation}
\overset{\circ}{c}_{N}(\omega,x)   =-\omega\sin\omega x+2\omega\sum
_{n=1}^{N}a_{n}\sum_{\text{odd }k=1}^{n}\binom{n}{k}\psi_{n-k}(x)\int_{0}%
^{x}t^{k}\sin\omega t\,dt +\frac{f^{\prime}(x)}{f(x)}c_N(\omega,x)
\label{cNcircle}
\end{equation}
and
\begin{equation}
\overset{\circ}{s}_{N}(\omega,x)    =\cos\omega x-2\sum_{n=0}^{N}b_{n}%
\sum_{\text{even }k=0}^{n}\binom{n}{k}\psi_{n-k}(x)\int_{0}^{x}t^{k}\cos\omega
t\,dt +\frac{f^{\prime}(x)}{f(x)}s_N(\omega,x).
\label{sNcircle}%
\end{equation}
The coefficients $\left\{  a_{n}\right\}  _{n=0}^{N}$ and $\left\{
b_{n}\right\}  _{n=1}^{N}$ are the same as in \eqref{cN} and \eqref{sN}, and
$b_{0}=h/2$.

The formulas for the approximations of the derivatives $v_{1}^{\prime}$ and
$v_{2}^{\prime}$ can be obtained with the aid of (\ref{ux}). Indeed, from (\ref{ux}) for
$v(y)=L^{-1}\left[  u(x)\right]  $  we have
\[
v_{y}=\frac{l_{y}}{\rho}u_{x}-\frac{\rho_{y}}{\rho}v=\sqrt{\frac{r}{p}}%
L^{-1}\left[  u^{\prime}(x)\right]  -\frac{\rho^{\prime}}{\rho}v.
\]
Considering $v=v_{1}(\omega,y)$ and $u=c(\omega,x;h)$ we obtain
\[
\frac{d}{dy}v_{1}(\omega,y)=\sqrt{\frac{r(y)}{p(y)}}L^{-1}\left[  c^{\prime
}(\omega,x;h)\right]  -\frac{\rho^{\prime}(y)}{\rho(y)}v_{1}(\omega,y).
\]
Similarly we have
\[
\frac{d}{dy}v_{2}(\omega,y)=\sqrt{\frac{r(y)}{p(y)}}L^{-1}\left[  s^{\prime
}(\omega,x;\infty)\right]  -\frac{\rho^{\prime}(y)}{\rho(y)}v_{2}(\omega,y).
\]
Now we can use the fact that the approximations for the functions $c^{\prime
}(\omega,x;h)$, $s^{\prime}(\omega,x;\infty)$ and $v_{1}(\omega,y)$,
$v_{2}(\omega,y)$ are given by (\ref{cNcircle}), (\ref{sNcircle}) and
(\ref{v1N}), (\ref{v2N}) respectively. Hence the approximations of the
derivatives $v_{1}^{\prime}$ and $v_{2}^{\prime}$ are given by the functions%
\begin{equation}
\overset{\circ}{v}_{1,N}(\omega,y)=\sqrt{\frac{r(y)}{p(y)}}L^{-1}\left[
\overset{\circ}{c}_{N}(\omega,x)\right]  -\frac{\rho^{\prime}(y)}{\rho
(y)}v_{1,N}(\omega,y) \label{v1Nc}%
\end{equation}
and
\begin{equation}
\overset{\circ}{v}_{2,N}(\omega,y)=\sqrt{\frac{r(y)}{p(y)}}L^{-1}\left[
\overset{\circ}{s}_{N}(\omega,x)\right]  -\frac{\rho^{\prime}(y)}{\rho
(y)}v_{2,N}(\omega,y) \label{v2Nc}%
\end{equation}
respectively.

In order to calculate $L^{-1}\left[  \overset{\circ}{c}_{N}(\omega,x)\right]
$ and $L^{-1}\left[  \overset{\circ}{s}_{N}(\omega,x)\right]  $ we use Theorem
\ref{relphikpsik}, relations (\ref{psin=Psin}) and the equality%
\[
L^{-1}\left[  \frac{f^{\prime}(x)}{f(x)}\right]  =\frac{1}{\rho(y)}\sqrt
{\frac{p(y)}{r(y)}}\left(  \frac{g^{\prime}(y)}{g(y)}+\frac{\rho^{\prime}%
(y)}{\rho(y)}\right)  .
\]
Thus,
\begin{equation*}
\begin{split}
L^{-1}\left[  \overset{\circ}{c}_{N}(\omega,x)\right]  &=-\frac{\omega}%
{\rho(y)}\sin\left(  \omega l(y)\right)  +\frac{2\omega}{\rho^{2}(y)}%
\sum_{n=1}^{N}a_{n}\sum_{\text{odd }k=1}^{n}\binom{n}{k}\Psi_{n-k}(y)\int
_{0}^{l(y)}t^{k}\sin\omega t\,dt\\
& \quad +\sqrt{\frac{p(y)}{r(y)}}\left(  \frac{g^{\prime}(y)}{g(y)}+\frac{\rho
^{\prime}(y)}{\rho(y)}\right)  v_{1,N}(\omega, y)
\end{split}
\end{equation*}
and
\begin{equation*}
\begin{split}
L^{-1}\left[  \overset{\circ}{s}_{N}(\omega,x)\right]  &=\frac{1}{\rho(y)}%
\cos\left(  \omega l(y)\right)  -\frac{2}{\rho^{2}(y)}\sum_{n=0}^{N}b_{n}%
\sum_{\text{even }k=0}^{n}\binom{n}{k}\Psi_{n-k}(y)\int_{0}^{l(y)}t^{k}%
\cos\omega t\,dt\\
&\quad +\sqrt{\frac{p(y)}{r(y)}}\left(  \frac{g^{\prime}(y)}%
{g(y)}+\frac{\rho^{\prime}(y)}{\rho(y)}\right)  v_{2,N}(\omega, y).
\end{split}
\end{equation*}

Substitution of these expressions into (\ref{v1Nc}) and (\ref{v2Nc}) gives us
the following result%
\begin{equation}\label{v1Ncc}
\begin{split}
\overset{\circ}{v}_{1,N}(\omega,y)  &=-\frac{\omega}{\rho(y)}\sqrt
{\frac{r(y)}{p(y)}}\sin\left(  \omega l(y)\right)+\frac{g^{\prime}(y)}{g(y)}v_{1,N}(\omega, y)\\
& \quad  +\frac{2\omega}{p(y)}%
\sum_{n=1}^{N}a_{n}\sum_{\text{odd }k=1}^{n}\binom{n}{k}\Psi_{n-k}(y)\int
_{0}^{l(y)}t^{k}\sin\omega t\,dt
\end{split}
\end{equation}
and
\begin{equation}\label{v2Ncc}
\begin{split}
\overset{\circ}{v}_{2,N}(\omega,y) &  =\frac{1}{\rho(y)}\sqrt{\frac
{r(y)}{p(y)}}\cos\left(  \omega l(y)\right) +\frac{g^{\prime}(y)}{g(y)}v_{2,N}(\omega, y)\\
&\quad  -\frac{2}{p(y)}\sum_{n=0}%
^{N}b_{n}\sum_{\text{even }k=0}^{n}\binom{n}{k}\Psi_{n-k}(y)\int_{0}%
^{l(y)}t^{k}\cos\omega t\,dt.
\end{split}
\end{equation}

\begin{remark}\label{Rmk NormalizedSolutions}
It is often convenient to have available the pair of solutions $V_{1}(\omega,y)$
and $V_{2}(\omega,y)$ of (\ref{SL AB}) satisfying the initial conditions
\[
V_{1}(\omega,y_0)=1,\quad V_{1}^{\prime}(\omega,y_0)=0
\]
and
\[
V_{2}(\omega,y_0)=0,\quad V_{2}^{\prime}(\omega,y_0)=1.
\]
Simple calculation gives us the following relations%
\[
V_{1}(\omega,y)=\rho(y_0)v_{1}(\omega,y)+\left(  \rho^{\prime}(y_0)\sqrt
{\frac{p(y_0)}{r(y_0)}}-h\rho(y_0)\right)  v_{2}(\omega,y)
\]
and
\[
V_{2}(\omega,y)=\rho(y_0)\sqrt{\frac{p(y_0)}{r(y_0)}}v_{2}(\omega,y).
\]
\end{remark}

Solution of Sturm-Liouville spectral problems for equation \eqref{SL AB} can be reduced to the search of zeros of a so-called characteristic function which can be written as a linear combination of the solutions $v_1$, $v_2$ and their derivatives. Numerical search of zeros of the characteristic function can benefit from the knowledge of the derivatives of $v_1$, $v_2$, $v_1'$ and $v_2'$ with respect to $\omega$, e.g., the Newton method can be used. One can differentiate the expressions \eqref{c cos} and \eqref{s sin} with respect to the variable $\omega$ and apply the constructed approximations of the transmutation operator to obtain the approximate derivatives and the corresponding error estimates. It appears that the final expressions obtained coincide with the termwise derivatives of \eqref{v1N}, \eqref{v2N}, \eqref{v1Ncc} and \eqref{v2Ncc}, cf., \cite{KTV NFT}, thus in order not to oversaturate the paper we provide only the approximations of $\partial_\omega v_1(\omega, y)$ and $\partial_\omega v_2(\omega, y)$:
\begin{equation}
\partial_\omega v_{1}(\omega,y)\approx -\frac{l(y)}{\rho(y)}\sin\left(  \omega l(y)\right)
-2\sum_{n=0}^{N}a_{n}\sum_{\text{even }k=0}^{n}\binom{n}{k}\Phi_{n-k}%
(y)\int_{0}^{l(y)}t^{k+1}\sin\omega t\,dt \label{v1Nomega}%
\end{equation}
and
\begin{equation}
\partial_\omega v_{2}(\omega,y)\approx \frac{1}{\omega}\left(  \frac{l(y)\cos\left(  \omega
l(y)\right)  }{\rho(y)}-v_{2,N}(\omega,y)+2\sum_{n=1}^{N}b_{n}\sum_{\text{odd }k=1}^{n}\binom
{n}{k}\Phi_{n-k}(y)\int_{0}^{l(y)}t^{k+1}\sin\omega t\,dt\right).  \label{v2Nomega}%
\end{equation}

\section{Numerical solution of spectral problems}\label{Sect Numerical}
\subsection{General scheme}
Consider a Sturm-Liouville spectral problem for equation \eqref{SL} on a segment $[A,B]$ with two general boundary conditions
\begin{equation}\label{BC}
    a_{i1} v(A)+a_{i2}v'(A)+a_{i3}v(B)+a_{i4}v'(B)=0,\qquad i=1,2,
\end{equation}
where $a_{ij}$ are arbitrary complex numbers. Moreover, $a_{ij}$ can be  sufficiently smooth functions of the spectral parameter.

The general scheme of application of the proposed method of analytic approximation of the transmutation operators to the solution of such spectral problems consists in the following, cf., 
\cite[Section 7.1]{KT AnalyticApprox}.
\begin{enumerate}
\item Compute $\widetilde l(y) = \int_A^y \{ r(s)/p(s)\}^{1/2}\,ds$, $y\in[A,B]$.
\item Find $y_0$ such that $\widetilde l(y_0) = \widetilde l(B)/2$, and let $l(y) = \widetilde l(y) - \widetilde l(y_0)$.
\item Find a non-vanishing on $[A,B]$ solution $g$ of \eqref{SLhom} satisfying the initial condition \eqref{PartSolIC}. For this the SPPS method \cite{KrPorter2010} or Remark \ref{Rmk Particular solution} can be used.
\item Compute the functions $\Phi_k(y)$ and $\Psi_k(y)$, $k=0,\ldots,N$ according to Definition \ref{psik}.
\item Compute the functions $\widetilde{\mathbf{c}}_{0}(y)$, $\widetilde{\mathbf{c}}_{m}(y)$ and $\widetilde{\mathbf{s}}_{m}(y)$, $m=1,\ldots, N$ according to Remark \ref{RmkNoLiouville}.
\item Find the approximation coefficients $\{a_n\}_{n=0}^N$ and $\{b_n\}_{n=1}^N$ from \eqref{Opt1} and \eqref{Opt2}.
\item Calculate the approximations $v_{1,N}(\omega, y)$ and $v_{2,N}(\omega, y)$ of solutions $v_1$ and $v_2$ by \eqref{v1N} and \eqref{v2N}. If necessary, calculate the approximations of the derivatives of these solutions by \eqref{v1Ncc} and \eqref{v2Ncc}.
\item The characteristic equation of the spectral problem can be obtained as usual, see, e.g., \cite[\S1.3]{Marchenko}. The nontrivial solution $c_1v_1+c_2v_2$ satisfies both boundary conditions \eqref{BC} if and only if the determinant of the obtained linear system of equations for $c_1$ and $c_2$ is equal to zero. Changing $v_1$ and $v_2$ and their derivatives by the corresponding approximations one obtains a function whose zeros approximate the eigenvalues of the spectral problem.
\end{enumerate}

It should be noted that all the steps of the proposed algorithm can be performed numerically, it is not necessary to know the exact particular solution $g$ or to evaluate the integrals defining functions $\Phi_k$ and $\Psi_k$ in a closed form. We refer the reader to \cite[Section 7.1]{KT AnalyticApprox} for the details of the numerical recursive integration (step 4) and of the solution of the approximation problems (step 6). It is worth  mentioning that changing the summation order in \eqref{v1N}, \eqref{v2N}, \eqref{v1Ncc} and \eqref{v2Ncc} can lead to a significant speed advantage due to the possibility to precompute the sums related to the point $y$ (they remain unchanged during solution of the spectral problem). We refer the reader to \cite{KKT} for the details.

\begin{remark}\label{Rmk Particular solution}
Suppose that the functions $p$ and $q$ are real valued and nonvanishing on $[A,B]$. Then a nonvanishing particular solution of \eqref{SLhom} can be constructed using the method described above. Indeed, writing equation \eqref{SLhom} as
\[
(pg')' = -\lambda q g
\]
we obtain an equation of the form \eqref{SL} with either $r=q$ and $\lambda = -1$ or $r=-q$ and $\lambda =1$ depending whether $p\cdot q<0$ or $p\cdot q>0$ on $[A,B]$. The proposed method can be applied to this equation using $g_0=1$ as a nonvanishing particular solution. If the normalized solution $V_1$ constructed as in Remark \ref{Rmk NormalizedSolutions} possesses a zero on $[A,B]$, a combination $V_1+iV_2$ can be taken. The zeros of the linearly  independent solutions $V_1$ and $V_2$ can not coincide and the approximations  \eqref{v1N} and \eqref{v2N} are real valued even for $\lambda=-1$ (hence for $\omega=i$) ensuring that the expression $V_1+iV_2$ is nonvanishing on the whole $[A,B]$.
\end{remark}

Since the point $y_0$ is distinct from either of the endpoints $A$ and $B$, the  algorithm in the proposed form does not made full use of the knowledge of the initial values of the solutions $v_1$ and $v_2$ given by \eqref{InitCondV1} and \eqref{InitCondV2}. One possible way to make use of these initial conditions was proposed in \cite[Remark 5.3]{KT AnalyticApprox}. It consists in working with the transmutation operators $T_c$ and $T_s$ instead of the transmutation operator $\mathbf{T}$. Then for the construction of the corresponding integral kernels it is sufficient to know the potential $Q(x)$ only on $[0,b]$, see Theorem \ref{TcTsMapsSolutions}. Therefore in the Liouville transformation we can take $y_0=A$, and the approximate kernels and the approximate solutions can be constructed in an exactly same way with the only change, one has to use $y_0=A$.

The knowledge of the  initial values \eqref{InitCondV1} and \eqref{InitCondV2} is especially useful when the first boundary condition \eqref{BC} involves values of the solution and of its derivative only at the point $A$, i.e., has the form
\begin{equation}\label{BC1}
    a_{11}v(A)+a_{12}v'(A)=0.
\end{equation}
In this case the proposed algorithm can be slightly simplified.
\begin{itemize}
\item[1--2.] Compute $l(y) = \int_A^y \{ r(s)/p(s)\}^{1/2}\,ds$, $y\in[A,B]$.
\item[3--7.] Perform the same steps as in the original algorithm using $y_0=A$.
\item[8.] Find a nontrivial linear combination $c_1v_1+c_2v_2$ satisfying the first boundary condition \eqref{BC1}. According to Remark \ref{Rmk NormalizedSolutions} one can take, e.g., $v=-a_{12}V_1+a_{11}V_2$. The characteristic equation of the problem is given by the second boundary condition. Replacing $v_1$ and $v_2$ and their derivatives by the corresponding approximations one obtains a function whose zeros approximate the eigenvalues of the spectral problem.
\end{itemize}

\subsection{Numerical examples}
\begin{example}\label{Ex1}
Consider the following spectral problem (c.f., Example \ref{Example16})
\begin{equation}\label{ExEq1}
\begin{cases}
-v''-\frac{v'}{y}+\left(  \frac{1}{4y^{2}}+\frac{2}{(y-\frac{1}{2})^{2}}\right)v=\omega^2 v,\\
v(1)=v(2)=0.
\end{cases}
\end{equation}
To make its consideration consistent with Example \ref{Example16} we applied the first proposed algorithm. A particular solution $g$ satisfying the initial conditions
\[
g(y_0)=\sqrt{\frac 23},\qquad g'(y_0)=\frac 53\sqrt{\frac 23}
\]
(which give us $h=2$ according to \eqref{PartSolIC} and \eqref{h}) was computed numerically using the SPPS representation \cite{KrPorter2010} in Matlab 2010 in machine precision. On step 4, 60 formal powers and 4001 uniformly distributed points were used to represent all the functions involved and to perform all the integrations modified 6 point Newton-Cottes integration rule was applied.

As was shown in \cite{KT Transmut}, the kernel of the transmutation operator $\mathbf{T}$ corresponding to \eqref{ExEq1} is a finite linear combination of generalized wave polynomials and hence the approximations \eqref{cN} and \eqref{sN} turn out to be exact solutions with the following coefficients
\[
a_0=1,\quad a_1=-\frac 32,\quad a_2=-\frac 34,\quad a_3=\frac 34,\quad b_1=\frac 12,\quad b_2=-\frac 34.
\]
We approximated numerically the functions $G_1$ and $G_2$ from \eqref{Opt1} and \eqref{Opt2} using the Remez algorithm (see \cite{KKTT} and references therein) and obtained the following coefficients
\begin{gather*}
a_0=1,\quad a_1=-1.5,\quad a_2=-0.75,\quad a_3=0.749999999999999,\\ b_1=0.500000000000001,\quad b_2= -0.750000000000001,
\end{gather*}
close to the exact ones.

The exact characteristic equation for the problem \eqref{ExEq1} can be written in the form
\begin{equation}\label{Ex1CharEq}
(4+3\omega^2)\sin\omega - 4\omega\cos\omega=0.
\end{equation}
This allows us to compare numerical results and the exact ones. The exact eigenvalues were found from \eqref{Ex1CharEq} by Wolfram Mathematica's function \texttt{FindRoot}. To find approximate eigenvalues we computed the characteristic function for the values of $\omega$ from 1 to 101 with the stepsize $0.02$, constructed a spline through the obtained points and found its zeros. Matlab's functions \texttt{spapi} and \texttt{fnzeros} were used. All 32 eigenvalues on this segment were found, and the maximal absolute error of the approximate eigenvalues was $1.4\cdot 10^{-14}$, while the relative error was less than $5\cdot 10^{-16}$.
\end{example}

\begin{example}\label{Ex2}
Consider the Bessel equation \cite[Eqn. 2.162]{Kamke} with the following boundary conditions
\begin{equation}\label{Ex2Eq}
    \begin{cases}
        (xu')'+xu = -\lambda\frac ux,\\
        u'(1)=u(4)=0.
    \end{cases}
\end{equation}
The characteristic equation of this problem has the form
\begin{equation}\label{Ex2CharEq}
    J_{i\omega}(4)\bigl(Y_{i\omega+1}(1)-Y_{i\omega -1}(1)\bigr) = Y_{i\omega}(4)\bigl(J_{i\omega +1}(1)-J_{i\omega-1}(1)\bigr),
\end{equation}
where $\lambda = \omega^2$, and the spectrum consists of one negative eigenvalue and of an infinite series of positive eigenvalues. In terms of normalized solutions introduced in Remark \ref{Rmk NormalizedSolutions} the characteristic equation has the form $V_1(\omega, 4) = 0$.

To this problem we applied the simplified algorithm as described in the end of the previous subsection. Note that the mapping $l(y)$ for this problem is not linear. For a nonvanishing particular solution $g$ we used the complex valued combination $g=g_1+ig_2$ of the solutions computed numerically from the SPPS representation (see \cite[Remark 5]{KrPorter2010}). As in the previous example, we used 60 formal powers and a uniform mesh of 4001 points to represent all the functions involved. The least squares method was used to solve the approximation problems \eqref{Opt1} and \eqref{Opt2}, where 28 functions were sufficient for an optimal machine-precision approximation. The corresponding $\varepsilon_1$ and $\varepsilon_2$ were $9\cdot 10^{-15}$ and $1.4\cdot 10^{-14}$. As it has been already observed in \cite[Example 7.5]{KT AnalyticApprox}, the proposed method may produce unreliable results for the values of $\omega$ close to the origin. To overcome this difficulty we combined the computed eigenvalues with those obtained from the SPPS method (known to work the best close to the origin). All 88 eigenvalues satisfying $\lambda \le 200^2$ were found with a maximum relative error $5\cdot 10^{-15}$. The computation time was 1.7 seconds. On Figure \ref{Ex2Fig} we present the absolute and the relative errors of the computed eigenvalues. The exact eigenvalues were obtained solving the characteristic equation \eqref{Ex2CharEq} with the aid of Mathematica 8.
\begin{figure}[tbh]
\centering
\includegraphics[
bb=132   316   477   475]{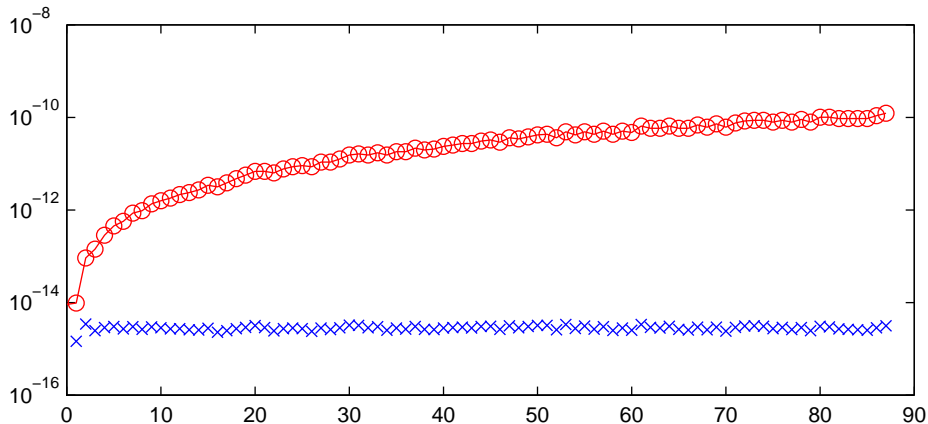}
\caption{Graphs of the absolute (red circles)
and relative (blue crosses) errors of the computed eigenvalues from Example \ref{Ex2}.
The eigenvalue index is on the abscissa.}
\label{Ex2Fig}
\end{figure}
\end{example}

\begin{example}\label{Ex3}
Consider the following problem (see \cite[Eqn. 2.273(11)]{Kamke})
\begin{equation}\label{Ex3Problem}
\begin{cases}
u''-2u'+u=-\lambda(y^2+1)u,\\
u(0)-u'(0)=0,\\
u(2)+u'(2)=0.
\end{cases}
\end{equation}
This problem is of the form \eqref{SL} with $p(y)=e^{-2y}$, $q(y)=-e^{-2y}$ and $r(y)=(y^2+1)e^{-2y}$. The characteristic equation of the problem has the form
\begin{equation}\label{Ex3CharEq}
    (1-i\omega)_1F_1\left(\frac{1+i\omega}4,\frac 12, 4i\omega\right)=
    (\omega^2-i\omega)_1F_1\left(\frac{5+i\omega}4,\frac 32, 4i\omega\right),
\end{equation}
where $\lambda = \omega^2$ and $_1F_1$ is the Kummer confluent hypergeometric function. In terms of the normalized solutions introduced in Remark \ref{Rmk NormalizedSolutions} the characteristic equation can be written as
\[
V_1(\omega, 2) + V_2(\omega, 2) + V'_1(\omega, 2) + V'_2(\omega, 2)=0.
\]

We applied the ``simplified'' algorithm. A non-vanishing particular solution was computed using the SPPS representation. The parameters for computation were chosen as in the previous examples. The precision achieved in solving the approximation problems \eqref{Opt1} and \eqref{Opt2} was $2.8\cdot 10^{-8}$ and $3.8\cdot 10^{-8}$ respectively (with 20 functions involved). We computed the first 100 eigenvalues and compared them with the exact ones obtained with the help of Wolfram Mathematica from the characteristic equation \eqref{Ex3CharEq}.
The maximum absolute and relative errors of the approximate eigenvalues were $4.7\cdot 10^{-7}$ and $8.5\cdot 10^{-9}$, respectively. On Figure \ref{Ex3Fig} we present the graphs of the errors.
\begin{figure}[tbh]
\centering
\includegraphics[
bb=132   316   477   475]{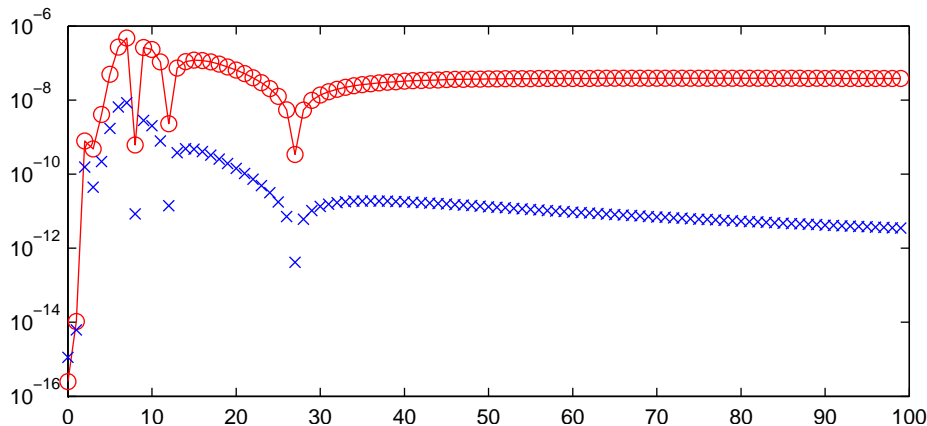}
\caption{Graphs of the absolute (red circles)
and relative (blue crosses) errors of the computed eigenvalues from Example \ref{Ex3}.
The eigenvalue index is on the abscissa.}
\label{Ex3Fig}
\end{figure}
\end{example}

\begin{remark}\label{Rmk ExamplesNoLiouville}
Although the first step in applying the proposed method is the Liouville transformation which apparently requires that the coefficients $p$ and $r$ be real-valued, a more thorough analysis leads to a surprising conclusion, that this condition might be superfluous. Indeed, the construction of functions $\rho$, $l$, $\Phi_n$, $\Psi_n$, $\widetilde{\mathbf{c}}_{m}$ and $\widetilde{\mathbf{s}}_{m}$ and Remark \ref{RmkNoLiouville} show that the approximation problems can be written in terms of quantities defined by the original equation \eqref{SL} only, without an explicit use of the Liouville transformation.
\end{remark}

In the last example we illustrate that the proposed method works under more general conditions on the coefficients $p$ and $r$.

\begin{example}\label{Ex4 complex}
Consider the following spectral problem
\begin{equation}\label{Ex4Problem}
\begin{cases}
u''=-\lambda e^{iy}u,\\
u'(0)=0,\qquad u(\pi)+u'(\pi)=0.
\end{cases}
\end{equation}
The characteristic function of the problem is given by
\[
F(\omega) =\omega  \Bigl(2K_1(2 \omega )\bigl(J_0(2 \omega )-i \omega  J_1(2 \omega )\bigr) +\pi
    I_1(2 \omega ) \bigl(-i J_0(2 \omega )-\omega  J_1(2 \omega )-Y_0(2
   \omega )+i \omega  Y_1(2 \omega )\bigr)\Bigr),
\]
where $\omega^2=\lambda$ and $I$, $J$, $K$ and $Y$ are the Bessel functions. In terms of the normalized solutions introduced in Remark \ref{Rmk NormalizedSolutions} the characteristic function can be written as $F(\omega)=V_1(\omega, \pi) + V_1'(\omega,\pi)$. Since the characteristic function is necessarily an analytic function of the variable $\omega$, we can apply the argument principle to localize its zeros, see \cite{Ying Katz}, \cite{Dellnitz Schutze Zheng} and our recent paper \cite{KTV} for details. After the zeros were localized within rectangles with the sides smaller than 0.1, we applied several Newton iterations to obtain approximate eigenvalues, the values of $F'(\omega)$ were computed using formulas \eqref{v1Nomega} and \eqref{v2Nomega}. On Figure \ref{Ex4Fig} we illustrate the work of the algorithm based on the argument principle. The absolute errors of the calculated eigenvalues satisfying $|\operatorname{Re}\omega|\le 50$ were less than $2.9\cdot 10^{-8}$ while the errors achieved in the approximation problems \eqref{Opt1} and \eqref{Opt2} were $2.7\cdot 10^{-11}$.
\begin{figure}[tbh]
\centering
\includegraphics[
bb=132   287   477   504]{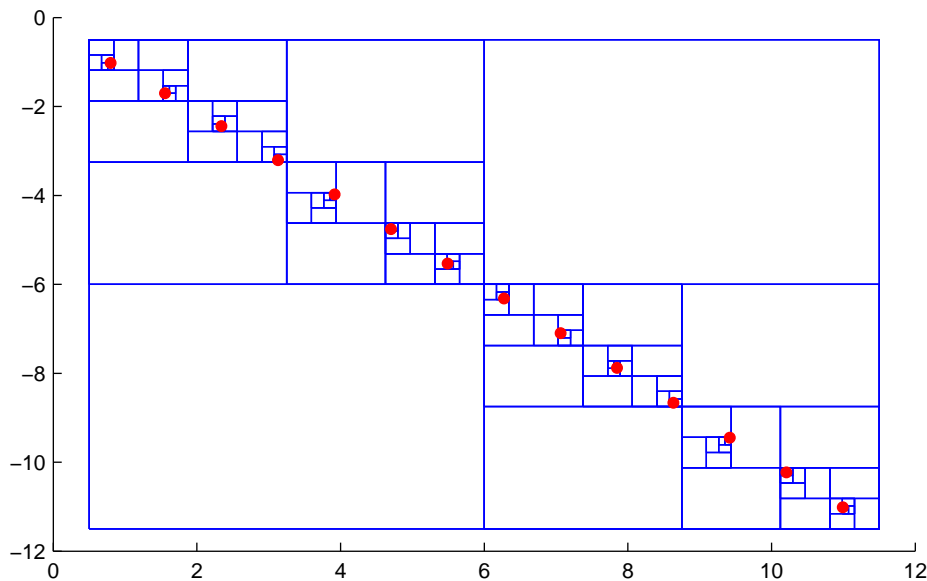}
\caption{Illustration to the work of the algorithm based on the argument principle in Example \ref{Ex4 complex}. Blue rectangles show the regions used to count the number of zeros on the subdivision step. Red circles mark the found eigenvalues.}
\label{Ex4Fig}
\end{figure}
\end{example}

Based on Remark \ref{Rmk ExamplesNoLiouville} and Example \ref{Ex4 complex}, we can formulate the following conjecture.
\begin{conjecture}
Theorem \ref{Thm Approxim Sols} holds and the proposed method works under weaker conditions than those required by Lemma \ref{definitionH}. Namely, it is sufficient that $p$ and $r$ be complex-valued nonvanishing on $(A,B)$, the requirement $p(y)$, $r(y)>0$ is superfluous.
\end{conjecture}

\end{document}